\documentclass[12pt]{article}
\usepackage{fullpage}
\usepackage{amsmath,amssymb,amsfonts,amsthm,amscd}
\usepackage{tikz,tikz-cd}
\usetikzlibrary{calc,shapes.geometric,decorations.markings}
\usepackage{mathtools,bm}
\usepackage{hyperref}
\usepackage[english]{babel}
\usepackage[left=3cm,right=3cm,top=3cm,bottom=4cm]{geometry}

\theoremstyle{plain}
\newtheorem{theorem}{Theorem}
\newtheorem{lemma}{Lemma}[section]
\newtheorem{proposition}[lemma]{Proposition}

\theoremstyle{definition}
\newtheorem{definition}[lemma]{Definition}
\newtheorem{example}[lemma]{Example}
\theoremstyle{remark}
\newtheorem*{remark}{Remark}

\newcommand{\bC}{\mathbb{C}}
\newcommand{\bP}{\mathbb{P}}
\newcommand{\bR}{\mathbb{R}}
\newcommand{\bZ}{\mathbb{Z}}
\newcommand{\cE}{\mathcal{E}}
\newcommand{\cF}{\mathcal{F}}
\newcommand{\cG}{\mathcal{G}}
\newcommand{\cL}{\mathcal{L}}
\newcommand{\cM}{\mathcal{M}}
\newcommand{\cO}{\mathcal{O}}
\newcommand{\fX}{\mathfrak{X}}
\newcommand{\pt}{\mathrm{pt}}
\newcommand{\vir}{\mathrm{vir}}
\renewcommand{\bar}{\overline}
\renewcommand{\tilde}{\widetilde}
\renewcommand{\vec}{\bm}
\newcommand{\cat}[1]{\mathsf{#1}}
\DeclareMathOperator{\Aut}{Aut}
\DeclareMathOperator{\ev}{ev}
\DeclareMathOperator{\Frac}{Frac}

\DeclareMathOperator{\Res}{Res}
\DeclareMathOperator{\SL}{SL}

\DeclareMathOperator{\val}{val}
\DeclarePairedDelimiter{\inner}{\langle}{\rangle}

\tikzset{%
  vertex/.style={shape=circle,fill=black,minimum size=4pt,inner sep=0},
  nonsing/.style={shape=circle,draw,minimum size=4pt,inner sep=0},
  baseline={([yshift=-0.8ex]current bounding box.center)}
}

\def\relcond#1#2{
  \begin{scope}[shift={#1}]
    \draw (0,0.05) -- (-0.4,0.8) -- (0,2) -- (0.4,1.2) -- cycle;
    \node[vertex] (#2) at (0,0) {};
  \end{scope}
}

\title{Self-duality in quantum K-theory}
\author{Henry Liu}
\date{\today}

\begin{document}

\maketitle

\begin{abstract}
  We describe an attempt to make quantum K-theory (of stable maps)
  more amenable to the self-duality/rigidity arguments of
  \cite{Okounkov2017} in quasimap theory, by twisting the virtual
  structure sheaf. For $\bP^n$ this twist produces invariants which
  are self-dual rational functions, but asymptotic analysis shows this
  is no longer the case for general GKM manifolds such as flag
  varieties. Such analysis is done via an explicit combinatorial
  description of localization for quantum K-theory on GKM manifolds,
  and Givental's adelic characterization.
\end{abstract}

\section{Introduction}

A K-theory class $\cF$ is {\it self-dual} if
\[ \cF = \cG - \hbar \cG^\vee \]
for some K-theory class $\cG$. A useful feature to demand from
enumerative theories of curves is that their deformation theory can be
made self-dual. This is the main feature of the theory of quasimaps to
Nakajima quiver varieties \cite{Okounkov2017}, and implies that, up to
some prefactors, K-theoretic integrals with no insertions over
quasimap moduli all take the form
\footnote{In quasimap theory it is important that we integrate the
  {\it symmetrized} virtual structure sheaf $\hat\cO^{\vir}$ instead
  of the unsymmetrized $\cO^{\vir}$ (see \cite[1.3.7, 1.4.2,
    3.2.7]{Okounkov2017}). The latter will produce expressions of the
  form
  \[ \prod_i \frac{1 - \hbar w_i^{-1}}{1 - w_i}, \]
  which are manifestly not balanced. The twist described in
  section~\ref{sec:cotangent-j-function} is our way of compensating
  for this discrepancy.}
\begin{equation} \label{eq:balanced-rational-function}
  \prod_i \frac{1 - \hbar w_i}{1 - w_i}
\end{equation}
for some equivariant weights $w_i$. Such rational functions are {\it
  bounded} in any equivariant limit $w^\pm \to \infty$, and we call
them {\it balanced}.

In cohomology, a very useful observation is that an (equivariant)
integral over a proper moduli is always a non-equivariant constant. In
K-theory this is no longer true; the result is a Laurent {\it
  polynomial} $p(\vec x)$ in equivariant variables $\vec x$. However
if in addition $p(\vec x)$ is bounded as $\vec x^\pm \to \infty$, then
in fact $p$ is constant. This is the case for invariants arising from
self-dual deformation theories. In general, this phenomenon is called
{\it rigidity}.

When rigidity arguments apply, $p(\vec x)$ can be computed in any
equivariant limit $x_i^\pm \to \infty$, which dramatically simplifies
localization formulas. For example, rigidity is the key tool in
identifying $q$-difference operators for quasimap vertices as qKZ
operators \cite{Okounkov2017}, and the computation of 1- and 2-leg
K-theoretic vertices in Donaldson--Thomas theory \cite{Kononov2019}.

We will investigate self-duality in quantum K-theory, i.e. K-theoretic
Gromov--Witten theory, on the moduli of stable maps $\bar\cM_{0,n}(X,
d)$, first defined in \cite{Lee2004}. When $X$ is a GKM manifold,
equivariant localization provides a combinatorial description of
K-theoretic descendant integrals
\[ \chi\left(\bar\cM_{0,n}(X, d), \prod_{i=1}^n \frac{\ev_i^* \vec\phi_i}{1 - q_i\cL_i}\right). \]
Localization in K-theory has more combinatorial complexity than in
cohomology. Section~\ref{sec:gkm-localization} reviews GKM manifolds
and describes the combinatorial algorithm for localization in
K-theory.

A given stable map $f\colon C \to X$ has tangent-obstruction theory
given by
\[ T^{\vir} \coloneqq \chi(C, f^*T_X) - \chi(\Omega_C(D), \cO_C), \]
which is evidently not self-dual. One can attempt to rectify this by
computing invariants such as the J-function with some {\it insertion}
$\cF$, such that the result is a balanced rational function of the
form \eqref{eq:balanced-rational-function} in every degree. Such an
insertion is described in section~\ref{sec:cotangent-j-function}. We
call the modified J-function the {\it cotangent} J-function and show
in section~\ref{sec:cotangent-j-function-for-Pn} that for $X = \bP^n$
it is self-dual.

Unfortunately, we show in section~\ref{sec:equivariant-infinities}
that the cotangent J-function is unlikely to be balanced for general
$X$, by analyzing its asymptotics in equivariant limits. The insertion
defining the cotangent J-function is a twist of the insertion
$S_{\hbar}^\bullet(R\pi_*\ev^*\Omega_X)^\vee$ producing stable maps
$f\colon C \to T^*X$. We suspect a further twist is necessary in
general.

This project was initially suggested by Andrei Okounkov. The author
wishes to thank him and Alexander Givental for many valuable
discussions, and for reading a draft of this paper. 

\section{Virtual localization}
\label{sec:gkm-localization}

Let $X$ be a GKM manifold with action by a torus $T$. In this section,
we describe (virtual) localization in K-theory for $\bar\cM_{0,n}(X,
d)$. This description of course admits the straightforward
generalization to $\bar\cM_{g,n}(X, d)$. However for genus $g > 0$, we
do not know how to compute the K-theoretic vertex. The K-theoretic
description should be compared to the cohomological description given
in \cite{Liu2017}.

\subsection{GKM graphs}

Let $n\coloneqq \dim X$ and $T = (\bC^\times)^m$. Following
\cite{Guillemin1999}, associated to the GKM manifold $X$ is its GKM
graph $(V, E)$, where:
\begin{enumerate}
\item (vertices) the vertices $V = V(X)$ are $T$-fixed points;
\item (edges) the edges $E = E(X)$ are $T$-invariant $\bP^1$'s;
\item (flags) pairs $(e, v) \in E \times V$ with $e$ incident to $v$
  are called flags, and $E_v$ denotes all flags at $v$.
\end{enumerate}
We will abuse notation and conflate graph-theoretic objects with the
geometric objects they represent. For instance, $N_{e/X}$ denotes the
normal bundle in $X$ of the $\bP^1$ represented by the edge $e$.

The GKM graph is decorated by the weight of the $T$-action on each
edge, recorded in a weight function $\vec w_X$: for an edge $e \in
E_v$ incident to $v$,
\[ \vec w_X(e, v) \in K_T(\pt) = \bZ[a_1^{\pm}, \ldots, a_m^{\pm}] \]
is the weight of the edge $e$ at the vertex $v$. Weights must satisfy
the following properties.
\begin{enumerate}
\item (GKM hypothesis) For every vertex $v$, any two distinct edges
  $e, e' \in E_v$ have independent weights. In K-theory, independence
  of weights means $\vec w_X(e, v) \neq \vec w_X(e', v)^s$ for any $s \in
  \bR$.

\item Let the edge $e$ connect vertices $v, v'$. Then:
  \begin{enumerate}
  \item $\vec w_X(e, v') = \vec w_X(e, v)^{-1}$;
  \item every edge $f_i \in E_v$ pairs with an edge $f'_i \in E_{v'}$
    at the opposite vertex, such that
    \[ \vec w_X(f'_i, v') = \vec w_X(f_i, v) \vec w_X(e, v)^{a_i} \]
    for some integer $a_i \in \bZ$. Order the edges such that $f_n =
    e$, so that $a_n = 2$.
  \end{enumerate}
  The second property corresponds to the decomposition $N_{e/X} =
  \bigoplus_{i=1}^{n-1} \cO_e(a_i)$ and $T_e = \cO_e(2)$.
\end{enumerate}

\begin{example}[Projective spaces] \label{ex:Pn-GKM-notation}
  The GKM graph of $\bP^n$ is the $n$-simplex. Label the vertices from
  $1$ to $n$. For an edge from $i$ to $j$ with $i < j$, the torus
  $(\bC^\times)^{n+1}$ acts with weight
  \[ a_{ij} \coloneqq a_i/a_j. \]
\end{example}

\begin{example}[Full flag varieties] \label{ex:GB-GKM-notation}
  Let $G$ be a semisimple Lie group with torus and Borel $T \subset B
  \subset G$. Then $T$ acts on the (generalized) {\it flag variety}
  $G/B$, and $T$-fixed points biject with the Weyl group. For example,
  the GKM graph of the variety $\SL_3/B$ of full flags in $\bC^3$ is
  
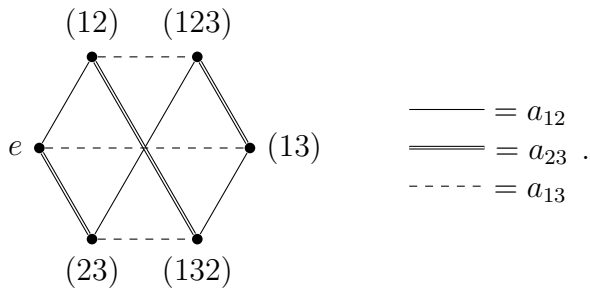
\[ \begin{tikzpicture}[scale=0.7]
    \coordinate[vertex,label=left:$e$] (e) at (180:2cm);
    \coordinate[vertex,label=above:$(12)$] (s12) at (120:2cm);
    \coordinate[vertex,label=right:$(13)$] (s13) at (0:2cm);
    \coordinate[vertex,label=below:$(23)$] (s23) at (240:2cm);
    \coordinate[vertex,label=above:$(123)$] (s123) at (60:2cm);
    \coordinate[vertex,label=below:$(132)$] (s132) at (-60:2cm);
    \begin{scope}
      \draw (e) -- (s12);
      \draw (s23) -- (s123);
      \draw (s13) -- (s132);
      \draw[double] (e) -- (s23);
      \draw[double] (s12) -- (s132);
      \draw[double] (s13) -- (s123);
      \draw[dashed] (e) -- (s13);
      \draw[dashed] (s12) -- (s123);
      \draw[dashed] (s23) -- (s132);
    \end{scope}
  \end{tikzpicture}
  \qquad
  \begin{array}{c}
    \begin{tikzpicture}\draw (0,0) -- (1,0);\end{tikzpicture} = a_{12} \\
    \begin{tikzpicture}\draw[double] (0,0) -- (1,0);\end{tikzpicture} = a_{23} \\
    \begin{tikzpicture}\draw[dashed] (0,0) -- (1,0);\end{tikzpicture} = a_{13}
  \end{array}. \]
\end{example}

\subsection{Kawasaki--Riemann--Roch}

The key distinction between localization on orbifolds in K-theory vs
in cohomology is that K-theory sees the geometry of twisted sectors,
and cohomology does not. In other words, if $\fX$ is an orbifold, we
must distinguish between an actual pushforward $\chi(\fX, \cE)$ and a
pushforward $\chi_{\cat{Sch}}(\fX, \cE)$ treating $\fX$ as a scheme.
The usual K-theoretic localization formula
\[ \chi_{\cat{Sch}}(\fX, \cE) = \sum_{F \subset \fX^T} \chi_{\cat{Sch}}(F, \cE|_F \cdot S^\bullet(N_{F/\fX}^{\vir})^\vee) \]
will fail if we choose the latter. In cohomology, this is because of
constant stacky factors like $|\Aut|$, but in K-theory the failure is
much more dramatic.

\begin{example}
  Let $\fX = \bP(m, 1)$ be a weighted projective line with an orbifold
  point at $0$. A torus $\bC^\times$ acts with weight $q^{1/m}$ around
  $0$ and weight $q^{-1}$ around $\infty$. Naively localizing gives
  \[ \chi_{\cat{Sch}}(\fX, \cO_\fX) = \frac{1}{1 - q^{-1/m}} + \frac{1}{1 - q}. \]
  This is not equal to the expected answer $1$, and worse, involves
  fractional powers of $q$.
\end{example}

This failure occurred because we forgot the stacky nature of the fixed
loci. Namely, the fixed point $0 \in \bP(m, 1)$ is actually a copy of
$B\mu_m$. K-theoretic integrals over it must involve taking
$\mu_m$-invariants, like in Lefschetz's fixed point formula: if
$\zeta_m$ is a primitive $m$-th root of unity, then
\begin{equation} \label{eq:lefschetz-averaging}
  \chi(B\mu_m, S^\bullet(N_{0/\fX})^\vee) = \frac{1}{m} \sum_{k=0}^{m-1} \frac{1}{1 - \zeta_m^k q^{-1/m}} = \frac{1}{1 - q^{-1}}.
\end{equation}
In the example, if we put this together with the contribution at
$\infty$, then we get the desired answer $1$.

The general principle is encapsulated in a Riemann--Roch formula for
Deligne--Mumford stacks. The idea, like in orbifold cohomology, is to
work with the inertia stack $I\fX = \bigsqcup_\mu \fX_\mu$. Here
$\fX_\mu$ are the connected components, which we view as embedded in
$\fX$ with normal bundle $N_{\fX_\mu/\fX}$. For example, if $\fX =
[V/G]$ is a global quotient,
\[ I[V/G] = \bigsqcup_{g \in \cat{Conj}(G)} [V^g/C(g)], \]
where we range over conjugacy classes and $C(g)$ is the centralizer of
$g$ in $G$. For simplicity, we define objects in this local case only.
Each component $\fX_\mu$ has a multiplicity
\[ m_\mu \coloneqq \ker\left(C(g) \to \Aut V^g\right). \]
An orbi-bundle $\cE$ on $[V^g/C(g)]$ decomposes into eigenbundles $\cE
= \sum_{k=0}^{m-1} \cE_k$ for the $g$-action, where $m$ is the order
of $g$. Define the virtual orbi-bundle
\begin{equation} \label{eq:trace-orbibundle-at-sector}
  t_\mu(\cE) \coloneqq \sum_{k=0}^{m-1} \zeta_m^k \cE_k.
\end{equation}

\begin{theorem}[\cite{Kawasaki1979}, \cite{Edidin2013}]
  \begin{equation} \label{eq:stacky-RR}
    \chi(\fX, \cE) = \sum_\mu \frac{1}{m_\mu} \chi_{\cat{Sch}}\left(\fX_\mu, t_\mu\left(\cE|_{\fX_\mu} \cdot S^\bullet(N_{\fX_\mu/\fX})^\vee\right)\right).
  \end{equation}
\end{theorem}

This is compatible with the virtual structure sheaf, i.e. with $\cE
\otimes \cO^{\vir}$ on the lhs and $N^{\vir}$ on the rhs
\cite{Tonita2014}.

\subsection{Fixed loci}

We briefly review the combinatorics of $T$-fixed loci in
$\bar\cM_{0,n}(X, d)$ for cohomological localization. Since fixed loci
in $\bar\cM_{0,n}(X, d)$ can be non-trivial orbifolds in general, for
K-theoretic localization the combinatorial data must be augmented by
the data of twisted sectors.

\begin{definition}[Cohomological trees]
  Given a fixed locus $F$, pick a stable map $f\colon (C, x) \to X$ in
  it and construct a decorated tree $\Gamma$ as follows (cf.
  \cite[section 9.2]{Cox1999}).
  \begin{enumerate}
  \item The vertices $V(\Gamma)$ represent {\it contracted} components
    of the source curve $C$, i.e. the connected components of
    $f^{-1}(X^T)$. Each vertex $v \in V(\Gamma)$ is labeled by any
    marked points it carries.
  \item The edges $E(\Gamma)$ represent {\it uncontracted} components
    of $C$. Each edge $e \in E(\Gamma)$ is labeled by the degree of
    the map $f$ on it:
    \[ \deg e \coloneqq \deg f|_e \in \bZ_{> 0}. \]
    In diagrams, to reduce clutter, we only label edges with their
    degree when $\deg > 1$.
  \end{enumerate}
\end{definition}

We will conflate vertices and edges in $\Gamma$ with their images in
(the GKM graph of) $X$. In particular, in diagrams, vertices will be
labeled by fixed points, in $V(X)$. If a vertex carries the $i$-th
marked point, we label it with a $\star_i$.

\begin{example}
  Consider $\bar\cM_{0,1}(\bP^1, d)$. We will only look at loci where
  the single marked point is at $0 \in \bP^1$. The only degree $1$ map
  $C \to \bP^1$ is the isomorphism, whose decorated tree is
  
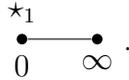
\[ \begin{tikzpicture}
    \coordinate[vertex,label=below:$0$,label=above:$\star_1$] (v1);
    \coordinate[vertex,right of=v1,label=below:$\infty$] (v2);
    \draw (v1) -- (v2);
  \end{tikzpicture}. \]
  In degree $2$, there are three different decorated trees:
  \begin{equation} \label{eq:P1-d2-trees}
    \begin{tikzpicture}
      \coordinate[vertex,label=below:$0$,label=above:$\star_1$] (v1);
      \coordinate[vertex,right of=v1,label=below:$\infty$] (v2);
      \draw (v1) -- node[label=$2$]{} (v2);
    \end{tikzpicture} \qquad
    \begin{tikzpicture}
      \coordinate[vertex,label=below:$0$,label=above:$\star_1$] (v1);
      \coordinate[vertex,right of=v1,label=below:$\infty$] (v2);
      \coordinate[vertex,right of=v2,label=below:$0$] (v3);
      \draw (v1) -- (v2) -- (v3);
    \end{tikzpicture} \qquad
    \begin{tikzpicture}
      \coordinate[vertex,label=below:$\infty$] (v1);
      \coordinate[vertex,right of=v1,label=below:$0$,label=above:$\star_1$] (v2);
      \coordinate[vertex,right of=v2,label=below:$\infty$] (v3);
      \draw (v1) -- (v2) -- (v3);
    \end{tikzpicture}.
  \end{equation}
  Note that there is a recursive structure. For a degree-$d$ tree,
  removing the marked vertex $v$ and its incident edge(s) produces a
  (collection of) trees in lower degree, whose marked vertex is a
  neighbor of $v$ and whose total degree is $d$ minus the degree(s) of
  the incident edge(s). Conversely, all degree-$d$ trees arise this
  way. This provides a recursive enumeration of $1$-pointed trees, as
  well as an important recursive structure in the J-function which we
  will see later.
\end{example}

Now we add stacky data. There are two kinds of automorphisms of
$\Gamma$: a non-trivial $\mu_d$ for an edge of degree $d$, and an
actual (structural) graph automorphism $S_k$ for every vertex with $k$
isomorphic legs. These two types of automorphisms {\it cannot} be
dealt with separately in localization, as we will see.

Suppose $\Gamma$ consists of just a single edge of degree $d$ with
weight $\vec w_X(e, v)$ in the target. Then $\Aut\Gamma = \mu_d$. From
the Kawasaki--Riemann--Roch formula \eqref{eq:stacky-RR}, we see that
in the $k$-th twisted sector of $B\mu_d$, the weight of the induced
$T$-action on the edge in the source is effectively $\zeta_d^k \vec
w_X(e, v)^{1/\deg e}$, living in a multiple cover of $K_T(\pt)$. Hence
the tree $\Gamma$ splits into $d$ different trees, for each of the $d$
sectors. Instead of just labeling the edge by $d$, we add a subscript
$k$ to the $d$. For example,
\begin{equation} \label{eq:cohom-to-K-theory-trees}
  \begin{tikzpicture}
    \coordinate[vertex,label=below:$a$] (v1);
    \coordinate[vertex,right of=v1,label=below:$b$] (v2);
    \draw (v1) -- node[label=$2$]{} (v2);
  \end{tikzpicture}
  \quad \leadsto \quad
  \begin{tikzpicture}
    \coordinate[vertex,label=below:$a$] (v1);
    \coordinate[vertex,right of=v1,label=below:$b$] (v2);
    \draw (v1) -- node[label=$2_0$]{} (v2);
  \end{tikzpicture} \quad
  \begin{tikzpicture}
    \coordinate[vertex,label=below:$a$] (v1);
    \coordinate[vertex,right of=v1,label=below:$b$] (v2);
    \draw (v1) -- node[label=$2_1$]{} (v2);
  \end{tikzpicture}.
\end{equation}

\begin{definition}[K-theoretic trees]
  A K-theoretic tree $\Gamma$ has the following additional data.
  \begin{enumerate}
  \item Vertices with $m$ groups of $k_i$ isomorphic legs have
    non-trivial automorphism group
    \[ S_{\vec k} \coloneqq S_{k_1} \times \cdots \times S_{k_m}. \]
    This data is implicit in the diagram and not explicitly written.

  \item Edges are labeled with their degree $d$ and {\it sector} $s$,
    with sector written as a subscript, i.e. $d_s$. The weight of the
    edge in the source is
    \[ \vec w_\Gamma(e, v) \coloneqq \zeta_d^s \vec w_X(e, v)^{1/d}. \]
    So a cohomological tree with non-trivial edge degrees generates
    many K-theoretic trees.
  \end{enumerate}
\end{definition}

It is important to write the new trees explicitly, because some new
trees may have extra structural automorphisms. For example, there are
degree-$4$ trees
\[ \begin{tikzpicture}
  \coordinate[vertex,left of=v1, label=below:$b$] (v3);
  \coordinate[vertex,label=below:$a$] (v1);
  \coordinate[vertex,right of=v1,label=below:$b$] (v2);
  \draw (v3) -- node[label=$2_0$]{} (v1) -- node[label=$2_0$]{} (v2);
\end{tikzpicture} \qquad
\begin{tikzpicture}
  \coordinate[vertex,left of=v1, label=below:$b$] (v3);
  \coordinate[vertex,label=below:$a$] (v1);
  \coordinate[vertex,right of=v1,label=below:$b$] (v2);
  \draw (v3) -- node[label=$2_0$]{} (v1) -- node[label=$2_1$]{} (v2);
\end{tikzpicture} \]
where the first tree has a structural $S_2$ automorphism and the
second tree does not. 

\subsection{Edges}

It remains to identify edge and vertex contributions to $S^\bullet
N_{\Gamma/X}^{\vir}$ for a K-theoretic tree $\Gamma$. There is no
additional K-theoretic complexity for edges once we work with
K-theoretic trees, so we describe edge contributions first.

Edge contributions come from deformations of the map $f$, which are
controlled by $\chi(C, f^*T_X)$. Let $e \in E(\Gamma)$ be an edge from
$v$ to $v'$. Then each term $\cO_e(a_i)$ in its normal bundle arises
from an edge $f_i \in E_v(X)$, and contributes to
$N_{\Gamma/\cM}^{\vir}$ the term
\[ \chi(e, \cO_e(a_i \deg e)) = \begin{cases} \;\,\displaystyle\sum_{k=0}^{a_i \deg e} \vec w_X(f_i, v) \vec w_\Gamma(e, v)^{-k} & a_i \ge 0 \\ \displaystyle\sum_{k=1}^{a_i \deg e - 1} \!\!\vec w_X(f_i, v) \vec w_\Gamma(e, v)^k & a_i < 0. \end{cases} \]
A similar computation holds for the tangent bundle $\cO_e(2)$, which
contains a single $T$-fixed weight $1$. After applying $S^\bullet$,
all cases unify if we introduce a {\it reduced} $q$-Pochhammer symbol
\[ (x; q)_d^{\text{red}} \coloneqq \frac{(x; q)_\infty^{\text{red}}}{(q^d x; q)_\infty^{\text{red}}}, \quad (x; q)_\infty^{\text{red}} \coloneqq \prod_{\substack{k \ge 0\\q^kx \neq 1}} (1 - q^k x) \]
to manually remove the $1$. Using this notation, the total
contribution of the edge $e$ is
\begin{equation} \label{eq:localization-edge-contribution}
  \frac{1}{\deg e} \frac{1}{(\vec w_X(e, v); \vec w_\Gamma(e, v)^{-1})_{2\deg e+1}^{\text{red}}} \prod_{i=1}^{n-1} \frac{1}{(\vec w_X(f_i, v); \vec w_\Gamma(e, v)^{-1})_{a_i \deg e + 1}}.
\end{equation}
The prefactor comes from averaging and is the same as the one in
\eqref{eq:stacky-RR}. For clarity, we separate the first term, coming
from the tangent bundle, and the other terms, coming from the normal
bundle.

\subsection{Vertices}

Vertex contributions to $S^\bullet N_{\Gamma/X}^{\vir}$ are more
difficult. Suppose a vertex $v \in V(\Gamma)$ has no marked points,
but has groups of isomorphic legs
\[ \Gamma_1 \cong \Gamma_2 \cong \cdots \cong \Gamma_{k_1}, \quad \Gamma_{k_1+1} \cong \cdots \cong \Gamma_{k_2}, \quad \ldots, \quad \cdots \cong \Gamma_{k_r}. \]
Then there is a permutation action by
\[ S_{\vec k} \coloneqq S_{k_1} \times \cdots \times S_{k_r}, \]
which means we must compute invariants over $[\bar\cM_{0,n}/S_{\vec
    k}]$ instead of over $\bar\cM_{0,n}$. In general, we can always
put such permutation actions on marked points in $\bar\cM_{g,n}(X,
d)$. This is Givental's permutation-equivariant quantum K-theory
\cite{Givental2015}, which has invariants
\[ \inner*{\cE}^{X,S_{\vec k}}_{g,n,d} \coloneqq \chi\left([\bar\cM_{g,n}(X, d) / S_n], \cE \otimes \cO^{\vir}\right) \]
for various sheaf insertions $\cE$. For the purpose of describing
virtual localization, we only need to understand the
permutation-equivariant theory of a point,
\[ \inner*{\cE}^{S_{\vec k}}_{0,n} \coloneqq \inner*{\cE}^{\pt, S_{\vec k}}_{0,n,d}, \]
which are precisely the vertex contributions for appropriate
insertions $\cE$.

We first explicitly describe the necessary insertion $\cE$. For
simplicity, suppose there is only one group of $k = k_1$ isomorphic
legs, permuted by a single $S_k$, which are the first $k$ out of $n$
legs. Let
\[ \vec \gamma_1 = \cdots = \vec \gamma_k, \; \vec \gamma_{k+1}, \ldots, \vec \gamma_n \in \Frac K_T(\pt) \]
be the contributions of the $n$ legs, which are of lower degree and
therefore we can assume have been computed already. Let
\[ w_i \coloneqq \vec w_\Gamma(e_i, v) \]
be the weight of the $i$-th leg, connected to $v$ by the edge $e_i$.
The connection, geometrically, is a node, whose deformation is
controlled by $w_i \cL_i$. There are also terms $(f|_v)^*T_X = T_vX$
arising from normalization of deformations of the map $f$. All
together, the integral we must compute at the vertex is
\begin{equation} \label{eq:localization-vertex-integral}
  \frac{1}{\wedge_{-1}^\bullet T_v X} \inner*{\prod_{i=1}^n \frac{\vec \gamma_i \cdot \wedge_{-1}^\bullet T_v X}{1 - w_i \cdot \cL_i}}^{S_k}_{0,n}
\end{equation}
(Compare with the vertex contributions in cohomology, e.g. in
\cite{Liu2017}.) If there were no permutation action, then the
numerator factors out of the integral and
\begin{equation} \label{eq:localization-non-equivariant-vertex-integral}
  \frac{1}{\wedge_{-1}^\bullet T_v X} \inner*{\prod_{i=1}^n \frac{\vec \gamma_i \cdot \wedge_{-1}^\bullet T_v X}{1 - w_i \cdot \cL_i}}_{0,n} \hspace{-5pt} = \;\left(\wedge_{-1}^\bullet T_v X\right)^{n-1} \prod_{i=1}^n \vec \gamma_i \cdot \inner*{\prod_{i=1}^n \frac{1}{1 - w_i \cdot \cL_i}}_{0,n}.
\end{equation}
Such non-permutation-equivariant integrals are the K-theoretic
analogue of gravitational descendants, and have an explicit formula.
(Note that, just as in the cohomological case, this formula continues
to work for $n = 1$ and $n = 2$, even though there is no actual
integral at the vertex in those cases.)

\begin{proposition}[\cite{Lee1997}]
  \begin{equation} \label{eq:YP-formula}
    \inner*{\prod_{i=1}^n \frac{1}{1 - w_i \cL_i}}_{0,n} \hspace{-5pt} = \;\left(1 + \sum_{i=1}^n \frac{w_i}{1 - w_i}\right)^{n-3} \prod_{i=1}^n \frac{1}{1 - w_i}.
  \end{equation}
\end{proposition}

In the permutation-equivariant case, we must again keep in mind
\eqref{eq:stacky-RR}. Characters of the $S_k$ action have different
values on different permuted legs, and so it is important that we do
{\it not} factor out the numerator like we did in
\eqref{eq:localization-non-equivariant-vertex-integral}.

\begin{example}
  Consider $[\bar\cM_{0,3}/S_2]$, where the $S_2$ permutes the first
  two marked points. Up to isomorphism, there is only one tree in
  $\bar\cM_{0,3}$, which we draw as
  
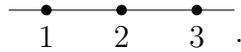
\[ \begin{tikzpicture}
    \draw (-0.5,0) -- (2.5,0);
    \coordinate[vertex,label=below:$1$] (v1) at (0,0);
    \coordinate[vertex,label=below:$2$] (v2) at (1,0);
    \coordinate[vertex,label=below:$3$] (v3) at (2,0);
  \end{tikzpicture}. \]
  When we integrate over $[\bar\cM_{0,3}/S_2]$, this tree splits into
  two trees --- one for the {\it unpermuted} sector, and one for the
  {\it permuted} sector:
  
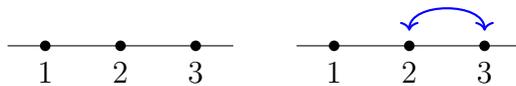
\[ \begin{tikzpicture}[baseline=0]
    \draw (-0.5,0) -- (2.5,0);
    \coordinate[vertex,label=below:$1$] (v1) at (0,0);
    \coordinate[vertex,label=below:$2$] (v2) at (1,0);
    \coordinate[vertex,label=below:$3$] (v3) at (2,0);
  \end{tikzpicture} \qquad
  \begin{tikzpicture}[baseline=0]
    \draw (-0.5,0) -- (2.5,0);
    \coordinate[vertex,label=below:$1$] (v1) at (0,0);
    \coordinate[vertex,label=below:$2$] (v2) at (1,0);
    \coordinate[vertex,label=below:$3$] (v3) at (2,0);
    \draw[thick,blue,<->] ($(v2)+(0,0.2)$) to [out=90,in=90] ($(v3)+(0,0.2)$);
  \end{tikzpicture} \]
  The arrow indicates that there is a non-trivial cyclic permutation
  on the marked points $2, 3$. This permutation is achieved via the
  map $z \mapsto -z$, and affects even the first marked point. For
  example,
  \begin{equation} \label{eq:M03-S2-formula}
    \inner*{\frac{1}{1 - q\cL_1}, \frac{1 - w}{1 - t\cL_2}, \frac{1 - w}{1 - t\cL_3}}_{0,3}^{S_2} = \frac{1}{2} \left(\frac{(1 - w)(1 - w)}{(1 - q)(1 - t)(1 - t)} + \frac{(1 - w)(1 \pmb{+} w)}{(1 \pmb{+} q)(1 - t)(1 \pmb{+} t)}\right).
  \end{equation}
\end{example}

In general, the symmetrization arising from cyclically permuting $r$
points can be written using Adams operations
\[ \Psi^r(a_i) \coloneqq a_i^r \in K_T(\pt) = \bZ[a_1^\pm, \ldots, a_m^\pm]. \]
For example, using the notation of
\eqref{eq:trace-orbibundle-at-sector}, if $\sigma$ contains a cycle of
length $r$, i.e. $r$ marked points on a single $\bP^1$ cyclically
permuted, then
\begin{equation} \label{eq:cyclic-permutation-formula}
  t_{\sigma} \left(\prod_{i=1}^r \frac{\vec \gamma}{1 - w\cL_i} \times \cdots\right) = \frac{\Psi^r(\vec \gamma)}{1 - w^r} \times \cdots.
\end{equation}
Every $S_k$-fixed tree is a collection of $\bP^1$'s connected by
nodes. For a given conjugacy class $[\sigma]$ in $S_k$, some $\bP^1$'s
carry cyclically permuted marked points (corresponding to cycles in
$\sigma$), and some carry non-permuted marked points. 
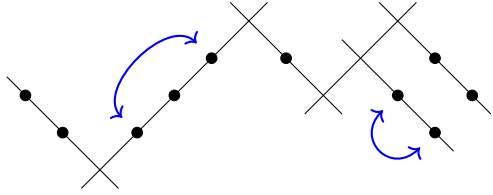
\[ \begin{tikzpicture}[rotate=-45,scale=0.7]
  \draw (-0.5,0) -- (2.5,0);
  \node[draw,vertex] at (0,0) {};
  \node[draw,vertex] at (1,0) {};
  \draw (2,-0.5) -- (2,4.5);
  \node[draw,vertex] (v1) at (2,1) {};
  \node[draw,vertex] at (2,2) {};
  \node[draw,vertex] (v2) at (2,3) {};
  \draw[thick,blue,<->] ($(v1)-(0.4,0)$) to [out=180,in=180] ($(v2)-(0.4,0)$);
  \draw (1.5,4) -- (4.5,4);
  \node[draw,vertex] at (3,4) {};
  \draw (4,3.5) -- (4,6.5);
  \draw (3.5,5) -- (6.5,5);
  \node[draw,vertex] (v3) at (5,5) {};
  \node[draw,vertex] (v4) at (6,5) {};
  \draw[thick,blue,<->] ($(v3)-(0,0.4)$) to [out=-90,in=-90,looseness=2] ($(v4)-(0,0.4)$);
  \draw (3.5,6) -- (6.5,6);
  \node[draw,vertex] at (5,6) {};
  \node[draw,vertex] at (6,6) {};
\end{tikzpicture} \]
The invariants in the two cases decouple, and we now know how to
handle both. In principle, this allows us to compute arbitrary vertex
integrals, but is combinatorially complicated.

In practice, there is a less combinatorially-involved algorithm for
permutation-equivariant vertices, arising from the adelic
characterization (cf. Theorems~\ref{thm:adelic-characterization},
\ref{thm:vertex-characterization}) of the {\it big} J-function of a
point in \cite[III, IX, X]{Givental2015}. This algorithm involves
recursively computing a change of parameters up to the desired degree.
However, it is not amenable to taking limits in equivariant
parameters, as we will do in Section~\ref{sec:equivariant-infinities}.
So we stick with the naive combinatorial description of the vertex.

\subsection{Algorithm}

To summarize, the contribution of a K-theoretic tree $\Gamma$ to an
invariant
\[ \chi\left(\bar\cM_{0,n}(X, d), \prod_{i=1}^n \frac{\ev_i^* \vec\phi_i}{1 - q_i \cL_i} \right) \]
is computed recursively as follows. It is the product of two types of
contributions: permutation-equivariant vertices, and usual
vertices/edges.

Let $V^{\text{perm}} \subset V(\Gamma)$ be the set of all vertices $v$
fixed by all automorphisms, but which have a non-trivial number of
isomorphic legs $\{\Gamma_i\}_{i=1}^r$ where $\Gamma_i$ appears $k_i$
times. These are the vertices where we must use
permutation-equivariant theory.
\begin{enumerate}
\item If $k_i > 1$, compute the contribution from $\Gamma_i$ using
  this algorithm recursively, and call it $\vec \gamma_i$.
\item If $k_i = 1$, so that the leg is not involved in any permutation
  action, set $\vec \gamma_i = 1$. Contributions from such legs will
  be included later.
\end{enumerate}
Let $I \subset \{1, \ldots, n\}$ be the set of marked points carried
by $v$. Then the total contribution from $v \in V^{\text{perm}}$ (cf.
\eqref{eq:localization-vertex-integral}) is
\begin{equation}
  \frac{1}{\wedge_{-1}^\bullet T_vX} \inner*{\prod_{i=1}^r \left(\frac{\vec \gamma_i \cdot \wedge_{-1}^\bullet T_vX}{1 - \vec w_\Gamma(e_i, v) \cdot \cL}\right)^{k_i} \prod_{i \in I} \vec\phi_i\big|_v}^{S_{\vec k}}_{0,\;\sum_{i=1}^r k_i + |I|}.
\end{equation}

Let $E^{\text{fixed}} \subset E(\Gamma)$ consist of all edges not
moved by any automorphism, and let $V^{\text{fixed}} \subset
V(\Gamma)$ be vertices incident to such edges but not in
$V^{\text{perm}}$. These are vertices and edges which are not involved
in permutation-equivariant machinery, whose contributions we compute
directly using \eqref{eq:localization-edge-contribution} for edges $e
\in E^{\text{fixed}}$ and
\eqref{eq:localization-non-equivariant-vertex-integral},
\eqref{eq:YP-formula} for vertices $v \in V^{\text{fixed}}$. (These
edges include the unpermuted legs for vertices $v \in V^{\text{perm}}$
above.)

The final value of the invariant is obtained by summing over all
K-theoretic trees $\Gamma$.

\section{J-function}

The J-function is an integral over $\bar\cM_{0,1}(X, d)$. The
$T$-fixed loci here splits into different components where the single
marked point is mapped to different fixed points $v \in V(X)$. Write
this splitting as
\[ \bar\cM_{0,1}(X, d)^T = \bigsqcup_{v \in V(X)} \bar\cM_{0,1}(X, d)_v. \]

\begin{definition}
  Fix $v \in V(X)$. The (small) {\it J-function} at $v$ is
  \[ J_{X,v}(q) \coloneqq 1 + \sum_{d > 0} \vec z^d J_{X,v}^{(d)}(q) \]
  where the degree $d$ term is
  \[ J^{(d)}_{X,v}(q) \coloneqq \frac{\wedge_{-1}^\bullet T_v X}{1 - q} \chi\left(\bar\cM_{0,1}(X, d)_v, \frac{\cO^{\vir}}{1 - q\cL}\right). \]
\end{definition}

Our definition differs from the standard one, e.g. in
\cite{Givental2015}, by the prefactors $\wedge_{-1}^\bullet T_v X$ and
$1/(1 - q)$. These prefactors arise naturally in the interpretation of
$J_X$ via the graph space construction, where instead of considering
maps $f\colon \bP^1 \to X$ of degree $d$, we consider their graphs,
which are maps $\Gamma_f\colon C \to \bP^1 \times X$ of bi-degree $(1,
d)$.
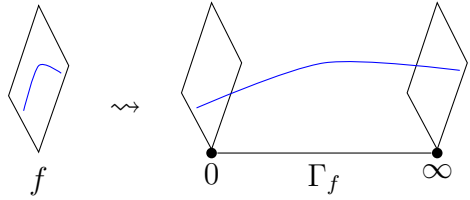
\[ \begin{tikzpicture}[baseline=0.55cm]
  \draw (0,0.05) -- (-0.4,0.8) -- (0,2) -- (0.4,1.2) -- cycle;
  \draw[blue] plot [smooth] coordinates { (-0.2,0.6) (0,1.2) (0.3,1.1) };
  \node[below] at (0,0) {$f$};
\end{tikzpicture}
\quad \leadsto \quad
\begin{tikzpicture}
  \relcond{(0,0)}{p};
  \relcond{(3,0)}{q};
  \draw (p) node[below]{$0$} -- (q) node[below]{$\infty$};
  \draw[blue] plot [smooth] coordinates { (-0.2,0.6) (1.5,1.2) (3.3, 1.1) };
  \node[below] at (1.5,0) {$\Gamma_f$};
\end{tikzpicture} \]
Let $\bC_q^\times$ act on the target $\bP^1$ factor. By localization,
the K-theoretic count of stable maps to graph space splits into two
J-functions:
\begin{alignat*}{2}
  \begin{tikzpicture}
    \relcond{(0,0)}{p};
    \relcond{(3,0)}{q};
    \draw (p) -- (q);
  \end{tikzpicture}
  \quad &= \quad 
  \begin{tikzpicture}
    \relcond{(0,0)}{p};
    \node[nonsing] (q) at (2,0) {};
    \draw (p) -- (q);
  \end{tikzpicture}
  &&\times
  \begin{tikzpicture}
    \node[nonsing] (p) at (0,0) {};
    \relcond{(2,0)}{q};
    \draw (p) -- (q);
  \end{tikzpicture} \\
  \chi\left(\bar\cM_{0,0}(\bP^1 \times X, (1, d)), \cO^{\vir}\right) &\propto \sum_{d_1 + d_2 = d} \; J_X^{(d_1)}(q) &&\times \quad J_X^{(d_2)}(q^{-1}).
\end{alignat*}
This idea was used very effectively in \cite{Givental2003} to compute
$J_{\bP^n}$. The extra prefactors come from the deformation theory of
the extra $\bP^1$ leg at the vertex corresponding to the marked point.

To compute the J-function, bluntly applying localization is not the
best method. The J-function has structure that other generating
functions of GW invariants do not have. To understand this structure,
we must allow for an arbitrary number of additional marked points,
each carrying some {\it input}.

\begin{definition}
  Let $\vec t = \sum_{m \ge 0} \vec t^{(m)} q^m$ with $\vec t_m \in
  K_T(X)$. The {\it big} J-function at $v$ with {\it input} $\vec t$
  is
  \[ J^{(d)}_{X,v}(q \; | \; \vec t) \coloneqq 1 + \frac{\vec t(q)}{1 - q} + \sum_{d > 0} \vec z^d J^{(d)}_{X,v}(q \; | \; \vec t), \]
  where the degree $d$ term is
  \[ J^{(d)}_{X,v}(q \; | \; \vec t) = \frac{\wedge_{-1}^\bullet T_v X}{1 - q} \sum_{n \ge 2} \chi\left(\bar\cM_{0,n+1}(X, d)_v, \frac{\cO^{\vir}}{1 - q\cL_{n+1}} \prod_{i=1}^n \sum_{m \ge 0} \ev_i^*(\vec t_i^{(m)}) \cL_i^m\right). \]
  We mostly care about the big J-function when $X = \pt$, in which
  case we omit writing $v$ and just write $J_{\pt}$.
\end{definition}

\subsection{Adelic characterization}

We review Givental's adelic characterization of the J-function, which
is our primary method for computing J-functions. In localization,
$1$-pointed K-theoretic trees index fixed loci in the integral for the
J-function. The marked point $\star$ is either on a vertex of valency
$1$ or of valency $> 1$. These two cases produce different poles in
$q$:
\begin{align*}
  \begin{tikzpicture}
    \node[draw,vertex,label=above:$\star$] at (0,0) {};
    \draw[postaction={decorate,decoration={markings,mark=at position 0.5 with {\arrow{>}}}}] (0,0) -- node[below]{$w$} (1,0);
    \node at (1.3,0) {$\cdots$};
  \end{tikzpicture} \; &= \; \frac{1}{1 - qw} \times \cdots \\
  \begin{tikzpicture}
    \node[draw,vertex,label=above:$\star$] at (0,0) {};
    \draw[postaction={decorate,decoration={markings,mark=at position 0.5 with {\arrow{>}}}}] (0,0) -- (1,0.7);
    \draw[postaction={decorate,decoration={markings,mark=at position 0.5 with {\arrow{>}}}}] (0,0) -- (1,0);
    \draw[postaction={decorate,decoration={markings,mark=at position 0.5 with {\arrow{>}}}}] (0,0) -- (1,-0.7);
    \node at (1.3,0.7) {$\cdots$};
    \node at (1.3,0) {$\cdots$};
    \node at (1.3,-0.7) {$\cdots$};
  \end{tikzpicture} \; &= \; \frac{1}{1 - q} \times \cdots.
\end{align*}

In the first case, we get a (simple) pole at $q = w^{-1}$. Let $e$
denote the edge, connecting the marked vertex $v$ with another vertex
$v'$. If we remove $v$ and $e$ and mark $v'$ instead, the resulting
$1$-pointed tree $\Gamma'$ contributes to $J^{(d-m)}_{X,v'}$, where $m
\coloneqq \deg(e)\beta(e)$. Let $w \coloneqq \vec w_\Gamma(e, v)$ be
the weight of the edge. Then
\begin{equation} \label{eq:adelic-characterization-edge}
  \Res_{q=w^{-1}} J^{(d)}_{X,v}(q) \frac{dq}{q} = \vec z^m J^{(d-m)}_{X,v'}(w^{-1}) \cdot E
\end{equation}
where $E$ is the contribution of the edge $e$ and vertex $v$.
Explicitly (cf. \eqref{eq:localization-edge-contribution}),
\begin{equation} \label{eq:adelic-edge-contribution}
  E = \frac{\wedge_{-1}^\bullet T_v X}{\deg e} \prod_{i=1}^n \frac{1}{(\vec w_X(f_i, v); \vec w_\Gamma(e, v)^{-1})^{\text{red}}_{a_i \deg e + 1}}.
\end{equation}

In the second case, we get poles (possibly non-simple) at $q$ equal to
roots of unity. Viewing the contributions of legs as inputs $\vec
t_i$, the value of the entire graph is a permutation-equivariant
vertex:
\begin{equation} \label{eq:adelic-characterization-vertex}
  \Res_{q=\zeta^{-1}} J_{X,v} \frac{dq}{q} = \inner*{\vec t_1(L), \ldots, \vec t_m(L), \frac{1}{1 - q\cL_{m+1}}}^{S_{\vec k}}_{0,m+1}.
\end{equation}

View $J_{X,v}$ as a meromorphic function of $q$. From these two cases,
we see $J_{X,v}$ can only have poles at $q = \xi$ where $\xi$ is a
root of either $1$ or some weight $\vec_X(e, v)$ for $e \in E_v(X)$.
The two cases place constraints on residues at these poles. In
addition, from the structure of virtual localization, $J_{X,v}$ has no
{\it regular} part and is therefore characterized by its residues at
poles. The converse also holds.

\begin{theorem}[{Adelic characterization, \cite[III, IV]{Givental2015}}]
  \label{thm:adelic-characterization}
  Suppose 
  \[ f(q) \in \Frac\left(K_T(\pt)\right)[[\vec z]] \]
  has no regular part and has poles only at $q$ equal to roots of
  unity or roots of weights of edges in $E_v(X)$. If in addition $f$
  satisfies \eqref{eq:adelic-characterization-edge} and
  \eqref{eq:adelic-characterization-vertex} for some inputs $\vec
  t_i$, then $f = J_{X,v}$.
\end{theorem}

For a given $f$, checking the recursion
\eqref{eq:adelic-characterization-edge} is usually straightforward,
though tedious. To check \eqref{eq:adelic-characterization-vertex}, we
need a better computational handle on values of $J_{\pt}(q \; | \;
\vec t)$ for various inputs $\vec t$. This is supplied by the action
of $q$-difference operators on the cone spanned by the $J_{\pt}(q \; |
\; \vec t)$, as well as an initial value computation when $\vec t$ is
constant in $q$.

\begin{theorem}[$D_q$-module structure for $X = \pt$, {\cite[I, IV]{Givental2015}}] \leavevmode \label{thm:vertex-characterization}
  \begin{enumerate}
  \item For $\nu \in K_T(\pt)$, the permutation-equivariant vertex is
    \[ J_{\pt}(q \; | \; \nu) = \exp_q\left(\frac{\nu}{1 - q}\right) \]
    where $\exp_q(x) \coloneqq \sum_{n \ge 0} x^n/(q; q)_n$ is the
    $q$-exponential function.

  \item If $f = \sum_{d \ge 0} \vec z^d f_d$ arises as $J_{\pt}(q\; |
    \; \vec t)$ for some inputs $\vec t$, then so do
    \[ \sum_{d \ge 0} \vec z^d f_d (\lambda q; q)_{\ell d}, \quad \sum_{d \ge 0} \vec z^d f_d \frac{1}{(\lambda q; q)_{\ell d}}, \]
    for any positive integer $\ell$ and parameter $\lambda$.
  \end{enumerate}
\end{theorem}

\subsection{Cotangent J-function}
\label{sec:cotangent-j-function}

We now describe a sheaf insertion which makes the J-function more
balanced.

\begin{definition}
  The {\it cotangent J-function} $\tilde J_X$ has terms
  \[ \tilde J_{X,v}^{(d)} \propto \chi\left(\bar\cM_{0,1}(X, d)_v, \frac{\cE}{1 - q\cL_1}\right) \]
  where $\cE$ is the sheaf
  \[ \cE \coloneqq S_\hbar^\bullet(-R\pi_*\ev^*\Omega_X). \]
  Note that GW theory with target $T^*X$ can be described as GW theory
  with target $X$ with insertion
  $S^\bullet(-R\pi_*\ev^*\Omega_X)^\vee$, which is $\cE$ twisted by
  $\det (-R\pi_*\ev^*\Omega_X)$.
\end{definition}

The insertion $\cE$ contributes extra terms in localization. Over $e
\in E(\Gamma)$, if there is a term $\cO_e(a)$ in $T_eX$, then there is
a corresponding contribution from $\cO_e(-a)$ in $\cE$. Hence the
total contribution of the edge $e$ is now
\begin{equation} \label{eq:hbar-edge-contribution}
  \frac{1}{\deg e} \prod_{i=1}^n \frac{(\hbar \vec w_X(f_i, v) \vec w_\Gamma(e, v)^{-1}; \vec w_\Gamma(e, v)^{-1})_{a_i \deg e - 1}}{(\vec w_X(f_i, v); \vec w_\Gamma(e, v)^{-1})_{a_i \deg e + 1}^{\text{red}}}.
\end{equation}
Every vertex $v \in V(\Gamma)$ also gains an extra factor
\[ \cE|_v = \wedge_{-\hbar}^\bullet T_v X. \]

\begin{example}
  Let $X = \bP^1$, with an action by $\bC^\times$ of weight $t$. We
  explicitly compute the $d=2$ term in the cotangent J-function at $0
  \in \bP^1$ via localization. This term has contributions from four
  K-theoretic trees (cf. \eqref{eq:P1-d2-trees},
  \eqref{eq:cohom-to-K-theory-trees}). For brevity, define
  \[ \{x_1, \ldots, x_n\} \coloneqq \prod (1 - x_i). \]
  The four contributions are:
  \begin{align*}
    \begin{tikzpicture}
      \coordinate[vertex,label=below:$0$,label=above:$\star$] (v1);
      \coordinate[vertex,right of=v1,label=below:$\infty$] (v2);
      \draw (v1) -- node[label=$2_0$]{} (v2);
    \end{tikzpicture}
    \quad &= \quad
    \frac{1}{2} \frac{\{\hbar t^{1/2}, \hbar, \hbar t^{-1/2}\} \{t^{-1/2}\}}{\{qt^{1/2}\} \{t, t^{1/2}, t^{-1/2}, t^{-1}\}} \\
    \begin{tikzpicture}
      \coordinate[vertex,label=below:$0$,label=above:$\star$] (v1);
      \coordinate[vertex,right of=v1,label=below:$\infty$] (v2);
      \draw (v1) -- node[label=$2_1$]{} (v2);
    \end{tikzpicture}
    \quad &= \quad
    \frac{1}{2} \frac{\{-\hbar t^{1/2}, \hbar, -\hbar t^{-1/2}\} \{-t^{-1/2}\}}{\{-qt^{1/2}\} \{t, -t^{1/2}, -t^{-1/2}, t^{-1}\}} \\
    \begin{tikzpicture}
      \coordinate[vertex,label=below:$0$,label=above:$\star$] (v1);
      \coordinate[vertex,right of=v1,label=below:$\infty$] (v2);
      \coordinate[vertex,right of=v2,label=below:$0$] (v3);
      \draw (v1) -- (v2) -- (v3);
    \end{tikzpicture}
    \quad &= \quad
    \frac{\{\hbar\} \{\hbar t^{-1}, t^{-1}\} \{\hbar\} \{t\}}{\{qt\} \{t, t^{-1}\} \{t^{-2}\} \{t, t^{-1}\}} \\
    \begin{tikzpicture}
      \coordinate[vertex,label=below:$\infty$] (v1);
      \coordinate[vertex,right of=v1,label=below:$0$,label=above:$\star$] (v2);
      \coordinate[vertex,right of=v2,label=below:$\infty$] (v3);
      \draw (v1) -- (v2) -- (v3);
    \end{tikzpicture}
    \quad &= \quad
    \frac{1}{2}\left(\frac{\{t^{-1}\}\{\hbar\}}{\{t, t^{-1}\}}\frac{\{t^{-1}\}\{\hbar\}}{\{t, t^{-1}\}} \frac{\{\hbar t, t\}}{\{q, t, t\}}\right. \\
    &\qquad\quad+ \left.\frac{\{t^{-1}\}\{\hbar\}}{\{t, t^{-1}\}}\frac{\{-t^{-1}\}\{-\hbar\}}{\{-t, -t^{-1}\}} \frac{\{-\hbar t, -t\}}{\{-q, t, -t\}}\right).
  \end{align*}
  The expression for the last term comes from the simplest case of the
  permutation-equivariant vertex \eqref{eq:M03-S2-formula}, over
  $[\bar\cM_{0,3}/S_2]$. Summing up all these contributions gives
  \[ \tilde J_{\bP^1,0}^{(2)} = \frac{\{t, \hbar t\}}{\{q\}} \cdot \frac{\{\hbar, \hbar q, \hbar qt\}}{\{t, qt, q^2, q^2t\}} = \frac{\{\hbar, \hbar t, \hbar q, \hbar qt\}}{\{q, qt, q^2, q^2t\}}. \]
\end{example}

\subsection{For projective space}
\label{sec:cotangent-j-function-for-Pn}

Using the adelic characterization, we can explicitly compute the
cotangent J-function for $X = \bP^n$. Notation is from
Example~\ref{ex:Pn-GKM-notation}.

\begin{proposition}
  The cotangent J-function for $\bP^n$ at the vertex $i$ is $(1 - q)$
  times the $q$-hypergeometric function
  \begin{equation} \label{eq:cotangent-j-function-for-Pn}
    I_i(q) \coloneqq \sum_{d \ge 0} z^d \prod_{j=0}^n \frac{(\hbar \frac{a_i}{a_j}; q)_d}{(q \frac{a_i}{a_j}; q)_d}.
  \end{equation}
\end{proposition}

\begin{proof}
  It is clear that $I$ automatically satisfies the vertex condition
  \eqref{eq:adelic-characterization-vertex} in the adelic
  characterization, purely because it is $q$-hypergeometric. It
  suffices then to verify the edge recursion
  \eqref{eq:adelic-characterization-edge} at the vertex $v = 0$; the
  verification at all other vertices is the same up to a change of
  variables.

  Let $e \in E(\Gamma)$ be an edge connecting vertices $0$ and $k$.
  Write $t_i \coloneqq a_0/a_i$ for brevity, so that $\vec w_X(e, v) =
  t_k$. In the target, this edge has tangent bundle $\cO(2)$ with
  linearization $(t_k, t_k^{-1})$, and normal bundles $\cO(1)$ with
  linearizations $(t_i, t_it_k^{-1})$ for all $i \neq k$. If it has
  degree $m$, then, using \eqref{eq:hbar-edge-contribution}, its
  contribution is
  \begin{align*}
    E
    &= \frac{\prod_{i=1}^n \{t_i, \hbar t_i\}}{m} \frac{(\hbar t_k^{1-1/m}; t_k^{-1/m})_{2m-1}}{(t_k; t_k^{-1/m})^{\text{red}}_{2m+1}} \prod_{i \neq k} \frac{(\hbar t_i t_k^{-1/m}; t_k^{-1/m})_{m-1}}{(t_i; t_k^{-1/m})_{m+1}} \\
    &= \frac{1}{m} \frac{(\hbar t_k; t_k^{-1/m})_m (\hbar; t_k^{-1/m})_m}{(t_k^{1-1/m}; t_k^{-1/m})_m^{\text{red}} (t_k^{-1/m}; t_k^{-1/m})_m} \prod_{i \neq k} \frac{(\hbar t_i; t_k^{-1/m})_m}{(t_i t_k^{-1/m}; t_k^{-1/m})_m}. 
  \end{align*}
  On the other hand, the residue is
  \[ \Res_{q=t_k^{-1/m}} I^{(d)}_0(q) \frac{dq}{q} = \frac{z^d}{m} \frac{(\hbar; t_k^{-1/m})_d}{(t_k^{-1/m}; t_k^{-1/m})_d} \frac{(\hbar t_k; t_k^{-1/m})_d}{(t_k^{1-1/m}; t_k^{-1/m})_d^{\text{red}}} \prod_{i \neq k} \frac{(\hbar t_i; t_k^{-1/m})_d}{(t_i t_k^{-1/m}; t_k^{-1/m})_d}. \]
  If we split each $q$-Pochhammer as $(x; q)_d = (x; q)_m (x q^m;
  q)_{d-m}$, this is manifestly the product of $E$ with
  \[ z^m I^{(d-m)}_k(t_k^{-1/m}) = z^d \frac{(\hbar; t_k^{-1/m})_{d-m} (\hbar t_k^{-1}; t_k^{-1/m})_{d-m}}{(t_k^{-1/m}; t_k^{-1/m})_{d-m} (t_k^{-1-1/m}; t_k^{-1/m})_{d-m}} \prod_{i \neq k} \frac{(\hbar t_i t_k^{-1}; t_k^{-1/m})_{d-m}}{(t_i t_k^{-1-1/m}; t_k^{-1/m})_{d-m}}. \]
  Hence \eqref{eq:adelic-characterization-edge} is satisfied and we
  are done by adelic characterization.
\end{proof}

An important feature of $I_i(q)$ is that it is {\it self-dual}. This
means that it is of the form
\[ \sum_{d \ge 0} z^d \prod_i \frac{(\hbar w_i; q)_d}{(q w_i; q)_d} \]
where the product is over some collection of equivariant weights
$\{w_i\}$. We view the $(\hbar w_i; q)_d$ in the numerator as ``dual''
to the $(q w_i; q)_d$ in the denominator. Such rational functions are
{\it bounded} as $w_i^{\pm 1} \to \infty$, which is the key technical
tool in the K-theoretic calculations of \cite{Okounkov2017}. 

In the limit $\hbar \to 0$, we recover the formula for $J_{\bP^n}$ in
\cite{Lee1999}.

\begin{remark}
  Givental points out that, if we take $J_{\bP^n}$ as known, then this
  result follows from a slight modification of the quantum Lefschetz
  theorem of \cite[XI]{Givental2015}. In general, the quantum
  Lefschetz theorem is concerned with obtaining $J_{T^*X}$, instead of
  $\tilde J_X$, from $J_X$. In the notation there, if we let $\tilde
  E$ denote $E \coloneqq T^*X$ but with the twist of our cotangent
  J-function, then the relevant $q$-difference operator for the
  K\"ahler parameter $z$ is of the form
  \[ \Gamma_{\tilde E} \sim \prod_{L \in E} \prod_{\ell\ge 0} \frac{1 - L q^\ell}{1 - L q^{D_L(d) z\partial_z} q^\ell} \]
  where $D_L(d) \coloneqq -\inner{c_1(L), d}$ is the degree with
  respect to $L$. When applied to $z^d$, the operator produces the
  extra factors $\prod_{L \in E} (L; q)_{D_L(d)}$. For example, for $X
  = \bP^n$ we have $T^*X = (n+1)\cO(-1) - \cO$, so that the extra
  factors are (up to a constant $(1 - \hbar)$ which can be neglected)
  \[ \prod_{j=0}^n (\hbar \cO(-1)/a_j; q)_d. \]
  These are precisely the extra factors in
  \eqref{eq:cotangent-j-function-for-Pn} (compared to $J_{\bP^n,i}$),
  once we restrict them to the $i$-th fixed point $p_i$, where
  $\cO(-1)|_{p_i} = a_i$.
\end{remark}

\subsection{Equivariant infinities}
\label{sec:equivariant-infinities}

For general GKM $X$, the cotangent J-function is not balanced.
Balanced rational functions are bounded in all equivariant limits
$a_i^\pm \to \infty$. We can compute asymptotics of $\tilde
J^{(d)}_{X,v}$ in such limits explicitly by analyzing the non-balanced
parts of each contribution from a tree $\Gamma$. The analysis will
show there exist equivariant limits where $\tilde J_X^{(d)} \to
\infty$.

\begin{definition}
  Given two rational functions $f, g$ of equivariant variables, write
  \[ f \sim g \]
  to mean that $O(f) = O(g)$ in any equivariant limit $a_i^\pm \to
  \infty$.
\end{definition}

Suppose $\Gamma$ is a {\it chain} between two points, i.e. with only
internal vertices of valency $2$:
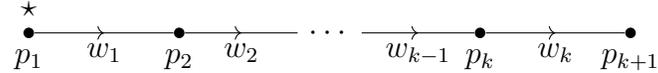
\[ \begin{tikzpicture}
  \node[vertex,label=below:$p_1$,label=above:$\star$] at (0,0) (v1) {};
  \node[vertex,label=below:$p_2$] at (2,0) (v2) {};
  \node at (4,0) (v3) {$\cdots$};
  \node[vertex,label=below:$p_k$] at (6,0) (v4) {};
  \node[vertex,label=below:$p_{k+1}$] at (8,0) (v5) {};
  \draw[postaction={decorate,decoration={markings,mark=at position 0.5 with {\arrow{>}}}}] (v1) -- node[below]{$w_1$} (v2);
  \draw[postaction={decorate,decoration={markings,mark=at position 0.5 with {\arrow{>}}}}] (v2) -- node[below]{$w_2$} (v3);
  \draw[postaction={decorate,decoration={markings,mark=at position 0.5 with {\arrow{>}}}}] (v3) -- node[below]{$w_{k-1}$} (v4);
  \draw[postaction={decorate,decoration={markings,mark=at position 0.5 with {\arrow{>}}}}] (v4) -- node[below]{$w_k$} (v5);
\end{tikzpicture} \]
Then we only need the edge recursion
\eqref{eq:adelic-characterization-edge} to compute its contribution.
Putting together \eqref{eq:adelic-edge-contribution} and
\eqref{eq:hbar-edge-contribution}, the contribution of an edge $e$
from $v$ to $v'$ to the recursion is
\[ E = \frac{\wedge_{-1}^\bullet T_v X \cdot \wedge_{-\hbar}^\bullet T_v X}{\deg e} \prod_{i=1}^n \frac{(\hbar \vec w_X(f_i, v) \vec w_\Gamma(e, v)^{-1}; \vec w_\Gamma(e, v)^{-1})_{a_i \deg e - 1}}{(\vec w_X(f_i, v); \vec w_\Gamma(e, v)^{-1})^{\text{red}}_{a_i \deg e + 1}}. \]
Using that
\begin{align*}
  \wedge^\bullet_{-\hbar} T_v X &= \prod_{i=1}^n (1 - \hbar \vec w_X(f_i, v)) \\
  \wedge^\bullet_{-1} T_{v'} X &= \prod_{i=1}^n (1 - \vec w_X(f_i, v) \vec w_\Gamma(e, v)^{-a_i \deg e})
\end{align*}
we can rearrange $E$ into a more suggestive form:
\[ E = \frac{1}{\deg e} \frac{\wedge_{-\hbar}^\bullet T_v X}{\wedge_{-1}^\bullet T_{v'} X} \prod_{i=1}^n \frac{(\hbar \vec w_X(f_i, v)\vec w_\Gamma(e, v)^{-1}; \vec w_\Gamma(e, v)^{-1})_{a_i \deg e - 1}}{(\vec w_X(f_i, v) \vec w_\Gamma(e, v)^{-1}; \vec w_\Gamma(e, v)^{-1})^{\text{red}}_{a_i \deg e - 1}} \]
The term in the product is balanced and contributes only a constant to
any equivariant limit. Hence
\[ E \sim \frac{\wedge_{-\hbar}^\bullet T_v X}{\wedge_{-1}^\bullet T_{v'} X}. \]
A chain of such edges produces a balanced product of such terms except
at the two endpoints of the chain:
\[ \prod_{i=1}^k \frac{\wedge_{-\hbar}^\bullet T_{p_i} X}{\wedge_{-1}^\bullet T_{p_{i+1}} X} \sim \frac{\wedge_{-\hbar}^\bullet T_{p_1} X}{\wedge_{-1}^\bullet T_{p_{k+1}} X}. \]

\begin{proposition} \label{prop:asymptotics-chain}
  The total contribution of the chain $\Gamma$ is
  \[ \Gamma \sim \frac{\wedge_{-\hbar}^\bullet T_p X}{\wedge_{-1}^\bullet T_q X} \cdot \prod_{i=1}^{k-1} \frac{1}{1 - w_i^{-1} w_{i+1}} \cdot (1 - w_k^{-1}). \]
\end{proposition}

\begin{proof}
  The only terms in $\Gamma$ we have neglected are vertex
  contributions. These are the extra terms appearing above. For a
  vertex of valency $1$ with (outgoing) edge of weight $w$, the
  contribution is $(1 - w)$. For a vertex of valency $2$ with
  (outgoing) edges of weights $w, w'$, the contribution is $1/(1 -
  ww')$.
\end{proof}

Now suppose $\Gamma$ is an arbitrary tree. This means we must take
into account general permutation-equivariant vertices $v$. Such
vertices are products of non-permuted terms, like in
\eqref{eq:YP-formula}, and cyclically permuted terms, like in
\eqref{eq:cyclic-permutation-formula}. Let $w_i$ be weights of edges
incident to $v$.
\begin{enumerate}
\item The first term in \eqref{eq:YP-formula} is a sum of balanced
  terms and can be disregarded. Hence the non-permuted case
  effectively contributes $\prod_i 1/(1 - w_i)$.
\item For the purpose of asymptotics, Adams operations in
  \eqref{eq:cyclic-permutation-formula} can be disregarded. Hence the
  cyclically permuted case also effectively contributes $\prod_i 1/(1
  - w_i)$.
\end{enumerate}
Diagrammatically, this means we can ``unglue'' vertices without
affecting asymptotics:
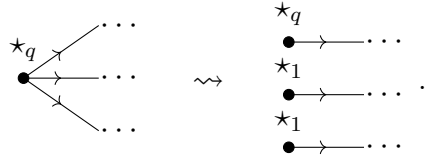
\[ \begin{tikzpicture}
  \node[draw,vertex,label=above:$\star_q$] at (0,0) {};
  \draw[postaction={decorate,decoration={markings,mark=at position 0.5 with {\arrow{>}}}}] (0,0) -- (1,0.7);
  \draw[postaction={decorate,decoration={markings,mark=at position 0.5 with {\arrow{>}}}}] (0,0) -- (1,0);
  \draw[postaction={decorate,decoration={markings,mark=at position 0.5 with {\arrow{>}}}}] (0,0) -- (1,-0.7);
  \node at (1.3,0.7) {$\cdots$};
  \node at (1.3,0) {$\cdots$};
  \node at (1.3,-0.7) {$\cdots$};
\end{tikzpicture}
\quad \leadsto \quad
\begin{tikzpicture}
  \node[draw,vertex,label=above:$\star_1$] at (0,0) {};
  \node[draw,vertex,label=above:$\star_q$] at (0,0.7) {};
  \node[draw,vertex,label=above:$\star_1$] at (0,-0.7) {};
  \draw[postaction={decorate,decoration={markings,mark=at position 0.5 with {\arrow{>}}}}] (0,0.7) -- (1,0.7);
  \draw[postaction={decorate,decoration={markings,mark=at position 0.5 with {\arrow{>}}}}] (0,0) -- (1,0);
  \draw[postaction={decorate,decoration={markings,mark=at position 0.5 with {\arrow{>}}}}] (0,-0.7) -- (1,-0.7);
  \node at (1.3,0.7) {$\cdots$};
  \node at (1.3,0) {$\cdots$};
  \node at (1.3,-0.7) {$\cdots$};
\end{tikzpicture}. \]
Here $\star_q$ reminds us there is a $1/(1 - q\cL)$ insertion at the
marked point, and $\star_1$ means we have to set $q=1$. This is so we
don't forget about the $\prod_i 1/(1 - w_i)$ and other contributions
from the original vertex.

Systematically ungluing all vertices in the tree $\Gamma$ produces a
collection of chains, whose asymptotics we already analyzed. Each
chain corresponds to a different leaf of $\Gamma$.

\begin{definition}
  If $v \in V(\Gamma)$ has incident edges of (outgoing) weights $w_1,
  w_2, \ldots$, define
  \[ c(v) \coloneqq \begin{dcases}
    1 - w_1 & \val(v) = 1 \\
    \frac{1}{1 - w_1w_2} & \val(v) = 2 \\
    \prod_i \frac{1}{1 - w_i} & \val(v) \ge 3.
  \end{dcases} \]
\end{definition}

\begin{theorem} \label{thm:asymptotics-general}
  Suppose $\Gamma$ has marked point $p$ and leaves $q_1, \ldots, q_k
  \in V(X)$. Then in any equivariant limit, the contribution of
  $\Gamma$ is
  \[ \Gamma \sim \prod_{i=1}^k \frac{\wedge_{-\hbar} T_p X}{\wedge_{-1} T_{q_i} X} \prod_{v \in V(\Gamma)} c(v). \]
\end{theorem}

Note that in both Proposition~\ref{prop:asymptotics-chain} and
Theorem~\ref{thm:asymptotics-general}, the main source of
divergence is the numerator $\wedge_{-\hbar} T_p X$. In general, if
$\dim X > 1$, there is no reason for trees with small numbers of
vertices to have any contribution canceling this divergence. 

In $\bP^n$, individual trees $\Gamma$ can have unbounded limits, but
cancellation occurs in the sum over trees. This is because $\bP^n$ has
Picard rank $1$ and its GKM graph is a complete graph. Given a tree
$\Gamma$, an edge from $v$ to $v'$ can be replaced by an edge from $v$
to $v''$ without changing the degree of $\Gamma$. This means a
divergence is averaged out across all $n$ vertices and cancels in the
end, like in Lagrange interpolation.

Theorem~\ref{thm:asymptotics-general} suggests it may be too much to
hope for some more complicated insertion which makes each contribution
$\Gamma$ balanced but which still produces an interesting J-function.

\subsection{For flag varieties}

We can examine the results of the previous section in the case of $X =
\SL_n/B$, using the notation of Example~\ref{ex:GB-GKM-notation}.

\begin{example}
  Let $X = \SL_3/B$, which has Picard rank $2$. Both $d = (k,0)$ and
  $d = (0, k)$ are analogous to the case of $X = \bP^1$ and produce
  self-dual J-functions. The first non-trivial degree is $d = (1,1)$.
  There are four trees contributing to $\tilde J^{(1,1)}_{X,e}$:
  \begin{alignat*}{2}
    \begin{tikzpicture}
    \coordinate[vertex,label=below:$e$,label=above:$\star$] (v1);
    \coordinate[vertex,right of=v1,label=below:$(13)$] (v2);
    \draw (v1) -- (v2);
    \end{tikzpicture}
    \quad &\sim \quad
    \frac{\{\hbar a_{12}, \hbar a_{13}, \hbar a_{23}\}}{\{qa_{13}, a_{12}^{-1}, a_{23}^{-1}\}} \quad &&\sim \quad a_{13} \\
    \begin{tikzpicture}
    \coordinate[vertex,label=below:$e$,label=above:$\star$] (v1);
    \coordinate[vertex,right of=v1,label=below:$(12)$] (v2);
    \coordinate[vertex,right of=v2,label=below:$(132)$] (v3);
    \draw (v1) -- (v2) -- (v3);
    \end{tikzpicture}
    \quad &\sim \quad
    \frac{\{\hbar a_{12}, \hbar a_{13}\}}{\{qa_{12}, a_{12}^{-1}, a_{23}\}} \quad &&\sim \quad a_{12} \\
    \begin{tikzpicture}
    \coordinate[vertex,label=below:$e$,label=above:$\star$] (v1);
    \coordinate[vertex,right of=v1,label=below:$(23)$] (v2);
    \coordinate[vertex,right of=v2,label=below:$(123)$] (v3);
    \draw (v1) -- (v2) -- (v3);
    \end{tikzpicture}
    \quad &\sim \quad
    \frac{\{\hbar a_{13}, \hbar a_{23}\}}{\{qa_{23}, a_{12}, a_{23}^{-1}\}} \quad &&\sim \quad a_{23} \\
    \begin{tikzpicture}
    \coordinate[vertex,label=below:$(23)$] (v1);
    \coordinate[vertex,right of=v1,label=below:$e$,label=above:$\star$] (v2);
    \coordinate[vertex,right of=v2,label=below:$(12)$] (v3);
    \draw (v1) -- (v2) -- (v3);
    \end{tikzpicture}
    \quad &\sim \quad
    \frac{\{\hbar a_{12}, \hbar a_{23}\}}{\{q, a_{12}, a_{23}\}} \quad &&\sim \quad 1.
  \end{alignat*}
  Hence $\tilde J_{X,e}^{(1,1)}$ cannot be balanced.
\end{example}

For flag varieties in general, the cotangent J-function should be
compared to the quasimap vertex for $T^*(\SL_n/B)$ in the following
way. This quasimap vertex was first explicitly computed in
\cite{Bullimore2015} as (up to constant factors):
\begin{equation}
  V(\vec z) \coloneqq \sum_{d_{i,j} \in C} \vec z^d \frac{\prod_{i=1}^{n-1} \prod_{j=1}^i \prod_{k=1}^i \frac{(\hbar a_{jk}; q)_{d_{i,j} - d_{i,k}}}{(qa_{jk}; q)_{d_{i,j} - d_{i,k}}}}{\prod_{i=1}^{n-1} \prod_{j=1}^i \prod_{k=1}^{i+1}\frac{(\hbar a_{jk}; q)_{d_{i,j} - d_{i+1,k}}}{(q a_{jk}; q)_{d_{i,j} - d_{i+1,k}}}}
\end{equation}
where $C$ imposes some stability conditions on the $\{d_{i,j}\}$ and
we set $d_{n,j} = 0$. Then \cite{Koroteev2018} checks that $V(q\vec
z/\hbar)$ is an eigenfunction of the trigonometric
Ruijsenaars--Schneider (tRS) system. In the limit $\hbar \to 0$, the
tRS Hamiltonians become $q$-Toda Hamiltonians. But \cite{Givental2003}
shows that the J-function for $G/B$ is a $q$-Toda eigenfunction. Since
our cotangent J-function is a twist of the J-function for
$T^*(\SL_n/B)$, it makes sense to try to compare it with the quasimap
vertex.

An important observation is that $V^{(d)}$ has a pole at $q = 0$. This
pole becomes a zero in $V(q\vec z/\hbar)$. From the structure of
localization, $J^{(d)}$ and $\tilde J^{(d)}$ cannot ever have zeros or
poles at $q = 0$. This is problematic. However, computations suggest
that residues at other poles of $\tilde J^{(d)}$ match with those of
$V^{(d)}$. In principle the edge recursion
\eqref{eq:adelic-characterization-edge} can be checked explicitly. But
we can no longer use Theorem~\ref{thm:vertex-characterization} to
check the vertex condition \eqref{eq:adelic-characterization-vertex},
because $\tilde J^{(d)}$ is no longer of $q$-hypergeometric form.
(Only $\SL_2/B = \bP^1$ and $\SL_3/B$ have $q$-hypergeometric terms in
their J-functions; the edge recursion is easily checked for the latter
as well.)

\phantomsection
\addcontentsline{toc}{section}{References}

\bibliographystyle{abbrv}
\bibliography{results}

\end{document}